\documentclass{article}
\usepackage{amssymb}

\usepackage{amsfonts}
\usepackage{amsmath}
\usepackage{geometry}

\newtheorem{theorem}{Theorem}[section]

\newtheorem{proposition}[theorem]{Proposition}

\newenvironment{proof}[1][Proof]{\textbf{#1.} }{\ \rule{0.5em}{0.5em}}
\input{tcilatex}
\begin{document}

\title{An extension theory and semidirect sums for nearrings}
\author{Stefan Veldsman \\
Nelson Mandela University (South Africa)\\
and\\
La Trobe University (Australia)}
\maketitle

\begin{abstract}
\noindent \noindent An extension theory for nearrings along the lines of the
Schreier extension for groups and the Everett extension for rings is given.
The semidirect sum of two nearrings is a special case of this theory.
Already there are many examples of purpose built nearring constructions in
the theory of nearrings which fall under this semidirect sum construction.
\end{abstract}

\begin{flushleft}
\textbf{AMS Subject Classification}: 16Y30

\textbf{Keywords: }nearring; semidirect sum; split epimorphism
\end{flushleft}

\section{General}

\noindent Nearrings are right distributive and need not be zero-symmetric.
More information on nearrings can be found in Pilz [11], Meldrum [9], Clay
[6] or Ferrero [8].

\bigskip

\noindent Bellas [3] defined a wreath product of (left) nearrings. For this,
a certain semidirect product of two nearrings was required which was duly
defined. Here we propose a more natural definition of a semidirect sum of
two nearrings which is a special case of a more general extension theory for
nearrings.The connection of this notion to that of a split epimorphism is
also given. We also give many examples of semidirect sums of nearrings. For
a general theory of semidirect products in universal algebra, consult
Facchini and Stanovsk\'{y} [7].

\bigskip

\noindent But we start by giving the more general extension theory for
nearrings along the lines of the Schreier extension for groups and the
Everett extension for rings. More information on these two constructions can
be found in Redei [12] or Petrich [10]. The Schreier extension for groups
will be given in some detail since it is required for the nearring extension.

\bigskip

\section{An extension theory for nearrings.}

\noindent Let $A$ and $B$ be two nearrings. A nearring $N$ is an \textit{%
extension of the nearring }$A$\textit{\ by }$B$ if $A$ is an ideal of $N$
and $N/A\cong B.$ On occasion there may be a need to refer to such an
extension more formally as a tripple $(i,N,p)$ where $i:A\rightarrow N$ is
an injective nearring homomorphism with $A\cong i(A)\lhd N$ and $%
p:N\rightarrow B$ is a surjective nearring homomorphism with $\ker (p)=i(A).$
Two extensions $(i_{1},N_{1},p_{1})$ and $(i_{2},N_{2},p_{2})$ of $A$ by $B$
are \textit{equivalent} if there is a nearring isomorphism $%
f:N_{1}\rightarrow N_{2}$ such that $f\circ i_{1}=i_{2}$ and $p_{2}\circ
f=p_{1}.$

\bigskip

\noindent For sets $X$ and $Y,$ $X^{2}$ denotes the cartesian product $%
X\times X,$ $Map(X,Y):=\{f\mid f:X\rightarrow Y$ a function$\}$ and $%
Map(X,X) $ is written as $Map(X).$ For a group $(G,+),$ $\Phi
:(G,+)\rightarrow (Aut(G,+),\circ )$ is the group homomorphism $\Phi
(a):=\Phi _{a}:G\rightarrow G$ with $\Phi _{a}(g)=a+g-a$ for all $a,g\in G.$
We want to describe an extension $N$ of $A$ by $B$ in terms of $A$ and $B$
and then show that all extensions of $A$ by $B$ are of this form. For this,
we will proceed as follows:

\noindent Step 1: Given two nearrings $A$ and $B,$ we construct a nearring $%
A\boxtimes B$ which is an extension of $A$ by $B.$ This nearring is called
an $E$\textit{-sum of }$A$\textit{\ and }$B$ and the construction will
depend on the existence of a quintuple of functions fulfilling a number of
requirements.

\noindent Step 2: Given two nearrings $A$ and $B$ and an extension $N$ of $A$
by $B,$ it is shown that $N$ is equivalent to an $E$-sum of $A$ and $B.$

\bigskip

\noindent \textbf{Step 1. }Let $A$ and $B$ be two nearrings. Suppose the
existence of five functions

$\bullet $\qquad $\psi :B\rightarrow Aut(A,+),\psi (b):=\psi
_{b}:A\rightarrow A;$

$\bullet $\qquad $\lbrack -,-]:B^{2}\rightarrow A;$

$\bullet $\qquad $\alpha :B\rightarrow Map(A^{2},A),\alpha (b):=\alpha
_{b}:A^{2}\rightarrow A;$

$\bullet $\qquad $\beta :B^{2}\rightarrow Map(A),\beta (b_{1},b_{2}):=\beta
_{b_{1},b_{2}}:A\rightarrow A;$ and

$\bullet $\qquad $\left\langle -,-\right\rangle :B^{2}\rightarrow A.$

\noindent These functions will be subject to some initial requirements to be
specified later. A quintuple of functions as above subject to the given
initial conditions, will be written as $E(\psi ,[-,-],\alpha ,\beta
,\left\langle -,-\right\rangle )$ and be referred to as the \textit{%
associated functions for the }$E$\textit{-sum} of $A$ and $B.$ The two
functions $[-,-]$ and $\left\langle -,-\right\rangle $ are called the 
\textit{factor system for the }$E$\textit{-sum}, the first for addition and
the second for the multiplication.

\noindent Let $A\boxtimes B$ be the cartesian product $A\times B=\{(a,b)\mid
a\in A,b\in B\}$ together with two binary operations on this set defined by:

\qquad $\qquad (a_{1},b_{1})+(a_{2},b_{2})=(a_{1}+\psi
_{b_{1}}(a_{2})+[b_{1},b_{2}],b_{1}+b_{2})$ and

\qquad $\qquad (a_{1},b_{1})(a_{2},b_{2})=(\alpha
_{b_{2}}(a_{1},a_{2})+\beta _{b_{1},b_{2}}(a_{2})+\left\langle
b_{1},b_{2}\right\rangle ,b_{1}b_{2}).$

\noindent By the definition of the associated functions, these two
operations are well-defined. For the construction of the underlying group of
the nearring $A\boxtimes B$ with respect to the above defined addition, we
will need the Schreier group extension theory. Everything below pertaining
to the construction of a group is not new and follows from the theory of the
Schreier group extension; in particular the first two theorems. It is given
in some detail here to facilitate the construction of the nearring on this
group.

\bigskip

\noindent The following two conditions on the functions $\psi $ and $[-,-]$
defined above will be necessary:

$(G_{1})$ $[b_{1},b_{2}]+[b_{1}+b_{2},b_{3}]=\psi
_{b_{1}}([b_{2},b_{3}])+[b_{1},b_{2}+b_{3}]$ for all $b_{1},b_{2},b_{3}\in
B; $ and

$(G_{2})$ $\psi _{b_{1}}\circ \psi _{b_{2}}=\Phi _{\lbrack
b_{1},b_{2}]}\circ \psi _{b_{1}+b_{2}}$ for all $b_{1},b_{2}\in B.$

\noindent A number of special cases of these two conditions is used in the
proof of the next result and is worth singling out.

$(i)$ $\psi _{b}([0,0])=[b,0]$ for all $b\in B$ (put $b_{2}=b_{3}=0$ in $%
(G_{1})).$

$(ii)$ $\psi _{0}([0,b])=[0,0]$ and hence $[0,b]=[0,b^{\prime }]$ for all $%
b,b^{\prime }\in B$ (put $b_{1}=b_{2}=0$ in $(G_{1})$ and use the
injectivity of $\psi _{0}).$

$(iii)$ $[b,-b]+[0,b]=\psi _{b}([-b,b])+[b,0]$ for all $b\in B$ (put $%
b_{1}=b_{3}=b$ and $b_{2}=-b$ in $(G_{1})).$

$(iv)$ $\psi _{0}(a)=[0,b]+a-[0,b]$ for all $b\in B,a\in A$ (put $b_{1}=0$
in $(G_{2})$ and remember that $\psi _{b}$ is a bijection).

\bigskip

\begin{theorem}
Let $(A,+)$ and $(B,+)$ be two groups and let $\psi $ and $[-,-]$ be the
functions $\psi :B\rightarrow Aut(A,+),\psi (b):=\psi _{b}:A\rightarrow A$
and $[-,-]:B^{2}\rightarrow A.$ Then $(A\boxtimes B,+)$ is a group if and
only if the functions $\psi $ and $[-,-]$ satisfy the two conditions $%
(G_{1}) $ and $(G_{2}).$

\noindent The group $(A\boxtimes B,+)$ has additive identity $(-[0,0],0)$
and the additive inverse of $(a,b)$ is $-(a,b)=(-(\psi
_{-b}(a)+[-b,b]+[0,0]),-b).$ The function $\xi :A\rightarrow A\boxtimes B$
defined by $\xi (a):=(a-[0,0],0)$ is an injective group homomorphism with $%
A\cong \xi (A)=(A,0):=\{(a,0)\mid a\in A\}$ which is a normal subgroup of $%
(A\boxtimes B,+)$ and $\eta :A\boxtimes B\rightarrow B$ with $\eta (a,b):=b$
is a surjective group homomorphism with $\ker (\eta )=(A,0).$
\end{theorem}

\begin{proof}
Suppose $(A\boxtimes B,+)$ is a group. Then

\qquad \qquad $\lbrack
(0,b_{1})+(0,b_{2})]+(0,b_{3})=(0,b_{1})+[(0,b_{2})+(0,b_{3})]$

\noindent which gives $(G_{1}).$ Again using the associativity of the
addition, we know

\qquad \qquad $\lbrack
(0,b_{1})+(0,b_{2})]+(a,0)=(0,b_{1})+[(0,b_{2})+(a,0)].$

\noindent This gives

\qquad \qquad $\lbrack b_{1},b_{2}]+\psi _{b_{1}+b_{2}}(a)+[b_{1}+b_{2},0]$

$\qquad \qquad =\psi _{b_{1}}(\psi _{b_{2}}(a))+\psi
_{b_{1}}([b_{2},0])+[b_{1},b_{2}]$

\noindent which can be written as

$\qquad \qquad \Phi _{\lbrack b_{1},b_{2}]}\circ \psi
_{b_{1}+b_{2}}(a)+[b_{1},b_{2}]+[b_{1}+b_{2},0]$

$\qquad \qquad =\psi _{b_{1}}\circ \psi _{b_{2}}(a)+\psi
_{b_{1}}([b_{2},0])+[b_{1},b_{2}].$

\noindent Using condition $(G_{1})$ with $b_{3}=0,$ we get $(G_{2}).$

\noindent Conversely, suppose the two conditions $(G_{1})$ and $(G_{2})$
hold. \ For the associativity we need

\qquad \qquad $\lbrack b_{1},b_{2}]+\psi
_{b_{1}+b_{2}}(a)+[b_{1}+b_{2},b_{3}]=\psi _{b_{1}}(\psi _{b_{2}}(a))+\psi
_{b_{1}}([b_{2},b_{3}])+[b_{1},b_{2}+b_{3}].$

\noindent Using $(G_{1})$ and $(G_{2})$ will validate the required equality.

\noindent The additive identity is $(-[0,0],0):$ For $(a,b)\in A\boxtimes B,$
$(a,b)+(-[0,0],0)=(a,b)$ using $(i)$ above since $\psi _{b}$ is a group
homomorphism. For $(-[0,0],0)+(a,b)=(a,b),$ we need $-[0,0]+\psi
_{0}(a)+[0,b]=a$ which follows from $(iv)$ and $(ii).$

\noindent Lastly, the additive inverse of $(a,b)\in A\boxtimes B,$ is $%
-(a,b)=(-(\psi _{-b}(a)+[-b,b]+[0,0]),-b).$ Indeed,

\qquad \qquad $(-(\psi _{-b}(a)+[-b,b]+[0,0]),-b)+(a,b)=(-[0,0],0)$

\noindent is clear and for

\qquad \qquad $(a,b)+(-(\psi _{-b}(a)+[-b,b]+[0,0]),-b)=(-[0,0],0),$

\noindent we need $a+\psi _{b}(-[0,0]-[-b,b]-\psi _{-b}(a))+[b,-b]=-[0,0].$

\noindent Now

$\qquad a+\psi _{b}(-[0,0]-[-b,b]-\psi _{-b}(a))+[b,-b]$

\qquad \qquad $=a-(\psi _{b}\circ \psi _{-b}(a)+\psi _{b}([-b,b])+\psi
_{b}([0,0])+[b,-b]$

\qquad \qquad $=a-(\Phi _{\lbrack b,-b]}(\psi
_{0}(a)+([b,-b]+[0,b]-[b,0])+[b,0])+[b,-b]$ by $(G_{2}),(iii)$ and $(i)$

\qquad \qquad $=a-([b,-b]+\psi _{0}(a)+[0,0])+[b,-b]$ by $(ii)$

\qquad \qquad $=a-([b,-b]+[0,0]+a)+[b,-b]$ by $(iii)$ and $(ii)$

\qquad \qquad $=-[0,0]$ as required.

\noindent That $\eta :A\boxtimes B\rightarrow B$ with $\eta (a,b):=b$ is a
surjective group homomorphism with kernel $(A,0)$ is clear; we show that $%
\xi :A\rightarrow A\boxtimes B$ defined by $\xi (a):=(a-[0,0],0)$ is a group
homomorphism:

\qquad $\xi (a_{1})+\xi (a_{2})$

$\qquad \qquad =(a_{1}-[0,0],0)+(a_{2}-[0,0],0)$

\qquad \qquad $=(a_{1}-[0,0]+\psi _{0}(a_{2})-\psi _{0}([0,0])+[0,0],0)$

$\qquad \qquad =(a_{1}+a_{2}-[0,0],0)$ by using $(ii)$ and $(iv)$

$\qquad \qquad =\xi (a_{1}+a_{2}).$
\end{proof}

\bigskip

\noindent The theorem above will actually be used subject to the initial
condition $[b,0]=[0,b]=0$ for all $b\in B.$ This assumption simplifies the
four consequences of $(G_{1})$ and $(G_{2})$ listed above to:

$(i)$ and $(ii):$ $\psi _{b}([0,0])=0$ for all $b\in B;$

$(iii)$ $[b,-b]=\psi _{b}([-b,b])$ for all $b\in B;$ and

$(iv)$ $\psi _{0}(a)=a$ for all $a\in A;$

\noindent and the theorem can be reformulated as:

\bigskip

\begin{theorem}
Let $(A,+)$ and $(B,+)$ be two groups and let $\psi $ and $[-,-]$ be the two
functions $\psi :B\rightarrow Aut(A,+),\psi (b):=\psi _{b}:A\rightarrow A$
and $[-,-]:B^{2}\rightarrow A$ with $[b,0]=[0,b]=0$ for all $b\in B.$Then $%
(A\boxtimes B,+)$ is a group if and only if the functions $\psi $ and $[-,-]$
satisfy the two conditions $(G_{1})$ and $(G_{2}).$

\noindent The group $(A\boxtimes B,+)$ has additive identity $(0,0)$ and
additive inverse $-(a,b)=(-(\psi _{-b}(a)+[-b,b]),-b).$ Moreover, the
function $\xi :A\rightarrow A\boxtimes B$ defined by $\xi (a):=(a,0)$ is an
injective group homomorphism with $A\cong \xi (A)=(A,0):=\{(a,0)\mid a\in
A\} $ which is a normal subgroup of $(A\boxtimes B,+)$ and $\eta :A\boxtimes
B\rightarrow B$ with $\eta (a,b):=b$ is a surjective group homomorphism with 
$\ker (\eta )=(A,0).$
\end{theorem}

\bigskip

\noindent The \textbf{initial conditions} required on the functions in the
quintuple $E(\psi ,[-,-],\alpha ,\beta ,\left\langle -,-\right\rangle )$ are
as follows.

\noindent For all $b,b_{1},b_{2}\in B,a_{1},a_{2}\in A:$

$\bullet \qquad \lbrack b,0]=[0,b]=0;$

$\bullet \qquad \alpha _{b}$ is linear in the first component (i.e., $\alpha
_{b}(a_{1}+a_{2},a)=\alpha _{b}(a_{1},a)+\alpha _{b}(a_{2},a)$);

$\bullet \qquad \alpha _{0}(a_{1},a_{2})=a_{1}a_{2};$

$\bullet \qquad \beta _{0,b}=0$ and $\beta _{b_{1},b_{2}}(0)=0;$ and

$\bullet \qquad \left\langle 0,b\right\rangle =0.$

\noindent We will consider four conditions on the function quintuple:

$(A_{1})$ $\alpha _{b_{3}}(\beta _{b_{1},b_{2}}(a_{2})+\left\langle
b_{1},b_{2}\right\rangle ,a_{3})+\beta
_{b_{1}b_{2},b_{3}}(a_{3})+\left\langle b_{1}b_{2},b_{3}\right\rangle $

$\qquad =\beta _{b_{1},b_{2}b_{3}}(\alpha _{b_{3}}(a_{2},a_{3})+\beta
_{b_{2},b_{3}}(a_{3})+\left\langle b_{2},b_{3}\right\rangle )+\left\langle
b_{1},b_{2}b_{3}\right\rangle ;$

\bigskip

$(A_{2})$ $\alpha _{b_{3}}(\alpha _{b_{2}}(a_{1},a_{2}),a_{3})=\alpha
_{b_{2}b_{3}}(a_{1},\alpha _{b_{3}}(a_{2},a_{3})+\beta
_{b_{2},b_{3}}(a_{3})+\left\langle b_{2},b_{3}\right\rangle );$

\bigskip

$(D_{1})$ $\alpha _{b_{3}}([b_{1},b_{2}],a_{3})+\beta
_{b_{1}+b_{2},b_{3}}(a_{3})+\left\langle b_{1}+b_{2},b_{3}\right\rangle $

$\qquad =\beta _{b_{1},b_{3}}(a_{3})+\left\langle b_{1},b_{3}\right\rangle
+\psi _{b_{1}b_{3}}(\beta _{b_{2},b_{3}}(a_{3}))+\psi
_{b_{1}b_{3}}(\left\langle b_{2},b_{3}\right\rangle
)+[b_{1}b_{3},b_{2}b_{3}];$

\bigskip

$(D_{2})$ $\alpha _{b_{3}}(\psi _{b_{1}}(a_{2}),a_{3})+\beta
_{b_{1},b_{3}}(a_{3})+\left\langle b_{1},b_{3}\right\rangle $

$\qquad =\beta _{b_{1},b_{3}}(a_{3})+\left\langle b_{1},b_{3}\right\rangle
+\psi _{b_{1}b_{3}}(\alpha _{b_{3}}(a_{2},a_{3})).$

\bigskip

\begin{theorem}
Given two nearrings $A$ and $B$ and the quintuple of functions $E(\psi
,[-,-],\alpha ,\beta ,\left\langle -,-\right\rangle )$ subject to the
initial conditions specified above. Then $(A\boxtimes B,+,\cdot )$ is a
nearring if and only if the function quintuple $E(\psi ,[-,-],\alpha ,\beta
,\left\langle -,-\right\rangle )$ satisfy the conditions $%
(G_{1}),(G_{2}),(A_{1}),(A_{2}),(D_{1})$ and $(D_{2}).$

\noindent If $A\boxtimes B$ is a nearring, then the function $\xi
:A\rightarrow A\boxtimes B$ defined by $\xi (a):=(a,0)$ is an injective
nearring homomorphism with $A\cong \xi (A)=(A,0):=\{(a,0)\mid a\in A\}$ an
ideal of $A\boxtimes B$ and $\eta :A\boxtimes B\rightarrow B$ with $\eta
(a,b):=b$ is a surjective nearring homomorphism with $\ker (\eta )=(A,0);$
hence $A\boxtimes B$ is an extension of $A$ by $B.$
\end{theorem}

\begin{proof}
Suppose $(A\boxtimes B,+,\cdot )$ is a nearring. Then $(G_{1})$ and $(G_{2})$
hold by the previous result. From the associativity of the multiplication,

$\qquad \qquad \lbrack
(0,b_{1})(a_{2},b_{2})](a_{3},b_{3})=(0,b_{1})[(a_{2},b_{2})(a_{3},b_{3})]$

\noindent holds which gives condition

$\qquad \qquad (A_{1})$ $\alpha _{b_{3}}(\beta
_{b_{1},b_{2}}(a_{2})+\left\langle b_{1},b_{2}\right\rangle ,a_{3})+\beta
_{b_{1}b_{2},b_{3}}(a_{3})+\left\langle b_{1}b_{2},b_{3}\right\rangle $

$\qquad \qquad \qquad =\beta _{b_{1},b_{2}b_{3}}(\alpha
_{b_{3}}(a_{2},a_{3})+\beta _{b_{2},b_{3}}(a_{3})+\left\langle
b_{2},b_{3}\right\rangle )+\left\langle b_{1},b_{2}b_{3}\right\rangle .$

\noindent The equality

\qquad \qquad $\lbrack
(a_{1},0)(a_{2},b_{2})](a_{3},b_{3})=(a_{1},0)[(a_{2},b_{2})(a_{3},b_{3})]$

\noindent and the initial conditions give $(A_{2}):$

\qquad \qquad $\alpha _{b_{3}}(\alpha _{b_{2}}(a_{1},a_{2}),a_{3})=\alpha
_{b_{2}b_{3}}(a_{1},\alpha _{b_{3}}(a_{2},a_{3})+\beta
_{b_{2},b_{3}}(a_{3})+\left\langle b_{2},b_{3}\right\rangle ).$

\noindent The multiplication is distributive from the right over the
addition, hence

\qquad \qquad $\lbrack
(0,b_{1})+(0,b_{2})](a_{3},b_{3})=(0,b_{1})(a_{3},b_{3})+(0,b_{2})(a_{3},b_{3}) 
$

\noindent which gives $(D_{1}):$

\qquad $\alpha _{b_{3}}([b_{1},b_{2}],a_{3})+\beta
_{b_{1}+b_{2},b_{3}}(a_{3})+\left\langle b_{1}+b_{2},b_{3}\right\rangle $

$\qquad \qquad =\beta _{b_{1},b_{3}}(a_{3})+\left\langle
b_{1},b_{3}\right\rangle +\psi _{b_{1}b_{3}}(\beta
_{b_{2},b_{3}}(a_{3}))+\psi _{b_{1}b_{3}}(\left\langle
b_{2},b_{3}\right\rangle )+[b_{1}b_{3},b_{2}b_{3}].$

\noindent Lastly, from the equality

\qquad \qquad $\lbrack
(a_{1},b_{1})+(a_{2},b_{2})](a_{3},b_{3})=(a_{1},b_{1})(a_{3},b_{3})+(a_{2},b_{2})(a_{3},b_{3}), 
$

\noindent use $(D_{1})$ to get $(D_{2}).$

\noindent Conversely, suppose the conditions $%
(G_{1}),(G_{2}),(A_{1}),(A_{2}),(D_{1})$ and $(D_{2})$ are satisfied. By the
previous result, we know $A\boxtimes B$ is a group. It is straightforward to
check that conditions $(A_{1})$ and $(A_{2})$ ensure that the multiplication
is associative and conditions $(D_{1})$ and $(D_{2})$ will give the right
distributivity. Hence $(A\boxtimes B,+,\cdot )$ is a nearring.

\noindent Suppose $(A\boxtimes B,+,\cdot )$ is a nearring. By the previous
result, the function $\xi :A\rightarrow A\boxtimes B$ defined by $\xi
(a):=(a,0)$ is an injective group homomorphism; we show it also preserves
the multiplication. For $a_{1},a_{2}\in A,$ $\xi (a_{1}a_{2})=(a_{1}a_{2},0)$
and $\xi (a_{1})\xi (a_{2})=(a_{1},0)(a_{2},0)=(\alpha
_{0}(a_{1},a_{2})+\beta _{0,0}(a_{2})+\left\langle 0,0\right\rangle
,0)=(a_{1}a_{2},0).$ Lastly, $\eta :A\boxtimes B\rightarrow B$ with $\eta
(a,b):=b$ is a surjective nearring homomorphism with $\ker (\eta )=(A,0).$
\end{proof}

\bigskip

\noindent \textbf{Step 2}. Given two nearrings $A$ and $B$ and an extension $%
N$ of $A$ by $B,$ which we write formally as $(i,N,p)$ where $i:A\rightarrow
N$ is $i(a)=a$ and $p:N\rightarrow B$ is a surjective nearring homomorphism
with $\ker (p)=A.$ We will show that this extension is equivalent to an $E$%
-sum $A\boxtimes B$ for some suitable choice of associated functions $E(\psi
,[-,-],\alpha ,\beta ,\left\langle -,-\right\rangle )$.

\bigskip

\noindent The nearring homomorphism $p:N\rightarrow B$ is surjective. Let $%
\sigma :B\rightarrow N$ be a choice function such that $p\circ \sigma =1_{B}$
subject to $\sigma (0)=0.$ Hence $\sigma (b)\in p^{-1}(b),$ and in
particular for $x\in N$ with $b:=p(x)\in B,$ we have $\sigma (b)=\sigma
(p(x))=x+a_{x}$ for some $a_{x}\in A.$ Now $x\in N$ has a unique
representation $x=x-\sigma (p(x))+\sigma (p(x))$ and $x-\sigma (p(x))\in A.$
Thus $f:N\rightarrow A\times B$ with $f(x):=(x-\sigma (b),b)$ where $b:=p(x)$
is a well-defined function. It is a bijection with $f^{-1}(a,b):=a+\sigma
(b).$ This means $f$ induces a nearring structure on $A\times B$ which is
described in what follows. Let $x_{1},x_{2}\in N,$ say $b_{j}=p(x_{j})$ and $%
a_{j}=x_{j}-\sigma (p(x_{j}))\in A$ for $j=1,2.$

\noindent Then $x_{1}+x_{2}=a_{1}+\sigma (b_{1})+a_{2}+\sigma (b_{2})$

$\qquad \qquad =a_{1}+(\sigma (b_{1})+a_{2}-\sigma (b_{1}))+(\sigma
(b_{1})+\sigma (b_{2}))$

$\qquad \qquad =a_{1}+\psi _{b_{1}}(a_{2})+[b_{1},b_{2}]+\sigma
(b_{1}+b_{2}) $ where

\noindent $\psi :B\rightarrow Aut(A,+),\psi (b):=\psi _{b}:A\rightarrow A$
is defined by $\psi _{b}(a)=\sigma (b)+a-\sigma (b)$ and

\noindent $\lbrack -,-]:B^{2}\rightarrow A$ is defined by $%
[b_{1},b_{2}]=\sigma (b_{1})+\sigma (b_{2})-\sigma (b_{1}+b_{2})$ for all $%
b,b_{1},b_{2}\in B,a\in A.$ We need to check that these functions are
well-defined. Because $A\lhd N,$ $\psi _{b}\in Aut(A,+).$ From $%
b_{1}+b_{2}=p(\sigma (b_{1}))+p(\sigma (b_{2}))=p(\sigma (b_{1})+\sigma
(b_{2}))=p(-a_{1}+x_{1}-a_{2}+x_{2})=p(x_{1}+x_{2}),$ we know $\sigma
(b_{1}+b_{2})=x_{1}+x_{2}+a_{3}$ for some $a_{3}\in A.$ Thus

$[b_{1},b_{2}]=\sigma (b_{1})+\sigma (b_{2})-\sigma (b_{1}+b_{2})$

$\qquad =-a_{1}+x_{1}-a_{2}+x_{2}-(x_{1}+x_{2}+a_{3})$

$\qquad =-a_{1}+(x_{1}-a_{2}-x_{1})+(x_{1}+x_{2})-a_{3}-(x_{1}+x_{2})\in A.$

\noindent Now $f(x_{1})+f(x_{2})=(x_{1}-\sigma (b_{1}),b_{1})+(x_{2}-\sigma
(b_{2}),b_{2})=(a_{1},b_{1})+(a_{2},b_{2})$ and

\qquad $f(x_{1}+x_{2})=(x_{1}+x_{2}-\sigma (p(x_{1}+x_{2})),p(x_{1}+x_{2}))$

$\qquad \qquad =(x_{1}+x_{2}-\sigma (b_{1}+b_{2})),b_{1}+b_{2})$

$\qquad \qquad =(a_{1}+\psi _{b_{1}}(a_{2})+[b_{1},b_{2}],b_{1}+b_{2}).$

\noindent Since we want $f(x_{1})+f(x_{2})=f(x_{1}+x_{2}),$ define addition
in $A\times B$ by

\begin{center}
$(a_{1},b_{1})+(a_{2},b_{2})=(a_{1}+\psi
_{b_{1}}(a_{2})+[b_{1},b_{2}],b_{1}+b_{2})$
\end{center}

\noindent for all $a_{1},a_{2}\in A$ and $b_{1},b_{2}\in B.$

\noindent For the multiplication,

$x_{1}x_{2}=(a_{1}+\sigma (b_{1}))(a_{2}+\sigma (b_{2}))$

$\qquad =a_{1}(a_{2}+\sigma (b_{2}))+\sigma (b_{1})(a_{2}+\sigma (b_{2}))$

$\qquad =[a_{1}(a_{2}+\sigma (b_{2}))]+[\sigma (b_{1})(a_{2}+\sigma
(b_{2}))-\sigma (b_{1})\sigma (b_{2})]+\sigma (b_{1})\sigma (b_{2})$

$\qquad =\alpha _{b_{2}}(a_{1},a_{2})+\beta
_{b_{1},b_{2}}(a_{2})+\left\langle b_{1},b_{2}\right\rangle +\sigma
(b_{1}b_{2})$ where

\noindent $\alpha :B\rightarrow Map(A^{2},A),\alpha (b):=\alpha
_{b}:A^{2}\rightarrow A$ is defined by $\alpha
_{b}(a_{1},a_{2})=a_{1}(a_{2}+\sigma (b));$

\noindent $\beta :B^{2}\rightarrow Map(A),\beta (b_{1},b_{2}):=\beta
_{b_{1},b_{2}}:A\rightarrow A$ is defined by $\beta _{b_{1},b_{2}}(a)=\sigma
(b_{1})(a+\sigma (b_{2}))-\sigma (b_{1})\sigma (b_{2});$ and

\noindent $\left\langle -,-\right\rangle :B^{2}\rightarrow A$ is defined by $%
\left\langle b_{1},b_{2}\right\rangle =\sigma (b_{1})\sigma (b_{2})-\sigma
(b_{1}b_{2})$

\noindent for all $a,a_{1},a_{2}\in A$ and $b,b_{1},b_{2}\in B.$ Since $%
A\lhd N,$ both the functions $\alpha $ and $\beta $ are well-defined. For
the function $\left\langle -,-\right\rangle ,$ we have

$b_{1}b_{2}=p(\sigma (b_{1}))p(\sigma (b_{2}))$

$\qquad \qquad =p(\sigma (b_{1})\sigma (b_{2}))$

$\qquad \qquad =p((-a_{1}+x_{1})(-a_{2}+x_{2}))$

$\qquad \qquad =p(-a_{1}(-a_{2}+x_{2})+x_{1}(-a_{2}+x_{2}))$

$\qquad \qquad
=p(-a_{1}(-a_{2}+x_{2})+x_{1}(-a_{2}+x_{2})-x_{1}x_{2}+x_{1}x_{2})$

$\qquad \qquad =p(x_{1}x_{2})$ since $A\lhd N.$

\noindent Thus $\sigma (b_{1}b_{2})=x_{1}x_{2}+a_{4}$ for some $a_{4}\in A.$
Then

$\left\langle b_{1},b_{2}\right\rangle =\sigma (b_{1})\sigma (b_{2})-\sigma
(b_{1}b_{2})$

$\qquad \qquad =(-a_{1}+x_{1})(-a_{2}+x_{2})-(x_{1}x_{2}+a_{4})$

$\qquad \qquad
=[-a_{1}(-a_{2}+x_{2})]+[x_{1}(-a_{2}+x_{2})-x_{1}x_{2}]+[x_{1}x_{2}-a_{4}-x_{1}x_{2}]\in A. 
$

\noindent Now $f(x_{1})f(x_{2})=(x_{1}-\sigma (b_{1}),b_{1})(x_{2}-\sigma
(b_{2}),b_{2})=(a_{1},b_{1})(a_{2},b_{2})$ and

\qquad $f(x_{1}x_{2})=(x_{1}x_{2}-\sigma (p(x_{1}x_{2})),p(x_{1}x_{2}))$

$\qquad \qquad \qquad =(x_{1}x_{2}-\sigma (b_{1}b_{2}),b_{1}b_{2})$

$\qquad \qquad \qquad =(\alpha _{b_{2}}(a_{1},a_{2})+\beta
_{b_{1},b_{2}}(a_{2})+\left\langle b_{1},b_{2}\right\rangle ,b_{1}b_{2}).$

\noindent Since we want $f(x_{1})f(x_{2})=f(x_{1}x_{2}),$ define
multiplication by

\begin{center}
$(a_{1},b_{1})(a_{2},b_{2})=(\alpha _{b_{2}}(a_{1},a_{2})+\beta
_{b_{1},b_{2}}(a_{2})+\left\langle b_{1},b_{2}\right\rangle ,b_{1}b_{2}).$
\end{center}

\noindent for all $a_{1},a_{2}\in A$ and $b_{1},b_{2}\in B.$

\noindent It can easily be checked that the five functions defined above
fulfill the initial requirements as specified in Step 1. Moreover, the
bijection $f$ with the addition and multiplication as defined above, gives a
nearring $(A\times B,+,\cdot )$ which is just an $E$-sum $A\boxtimes B$ of $%
A $ and $B$ with associated functions $E(\psi ,[-,-],\alpha ,\beta
,\left\langle -,-\right\rangle ).$ By the theorems above, these functions
satisfy the conditions $(G_{1}),(G_{2}),(A_{1}),(A_{2}),(D_{1})$ and $%
(D_{2}).$ We thus have:

\bigskip

\begin{theorem}
For two nearrings $A$ and $B,$ let $N$ be an extension of $A$ by $B,$
written formally as $(i,N,p)$ where $i:A\rightarrow N$ is $i(a)=a$ and $%
p:N\rightarrow B$ is a surjective nearring homomorphism with $\ker (p)=A.$
Then the extension $(i,N,p)$ is equivalent to an $E$-sum $(\xi ,A\boxtimes
B,\eta )$ for some suitable choice of associated functions $E(\psi
,[-,-],\alpha ,\beta ,\left\langle -,-\right\rangle )$ fulfilling the
initial conditions as well as the conditions $%
(G_{1}),(G_{2}),(A_{1}),(A_{2}),(D_{1})$ and $(D_{2}).$
\end{theorem}

\begin{proof}
In view of the preceeding, we only have to show the equivalence of the two
extensions $(i,N,p)$ and $(\xi ,A\boxtimes B,\eta ).$ We have $%
f:N\rightarrow A\boxtimes B,$ $f(x)=(x-\sigma (p(x)),p(x))$ is a nearring
isomorphism, $i(a)=a,$ $\xi (a)=(a,0)$ and $\eta (a,b)=b$ which give $f\circ
i=\xi $ and $\eta \circ f=p.$
\end{proof}

\bigskip

\noindent Two special cases of an $E$-sum $A\boxtimes B$ with associated
functions $E(\psi ,[-,-],\alpha ,\beta ,\left\langle -,-\right\rangle )$
fulfilling the initial conditions are:

(1) When the factor systems $[-,-]$ and $\left\langle -,-\right\rangle $ are
both $0$ and $\psi :(B,+)\rightarrow (Aut(A,+),\circ )$ is a group
homomorphism, then the conditions $(G_{1})$ and $(G_{2})$ are trivially
fulfilled and $A\boxtimes B$ is the (outer) semidirect sum of $A$ and $B$ if
and only if the conditions $(A_{1}),(A_{2}),(D_{1})$ and $(D_{2})$ hold.
This case will be discussed in more detail in the next section.

(2) When the factor systems $[-,-]$ and $\left\langle -,-\right\rangle $ are
both $0,$ then $A\boxtimes B$ is a direct sum of $A$ and $B$ if and only if
for all $b,b_{1},b_{2}\in B,\psi _{b}=1_{A},\alpha _{b}=\alpha _{0}$ and $%
\beta (b_{1},b_{2})=0.$

\bigskip

\section{The semidirect sum of two nearrings.}

\noindent A nearring $N$ is the (\textit{internal}) \textit{semidirect sum}
of two nearrings $A$ and $B,$ written as $A\rtimes B,$ if $A$ is an ideal of 
$N,$ $B$ is a subnearring of $N,$ $N=A+B,$ $A\cap B=0$ and $N/A\cong B.$ In
particular, this means the group $(N,+)$ is a semidirect sum of the groups $%
(A,+)$ and $(B,+).$ We should emphasize the obvious: the order in which the
nearrings $A$ and $B$ are mentioned is important. Recall, a nearring
homomorphism $f:N\rightarrow B$ is a \textit{split epimorphism }(=\textit{\
retraction})\textit{\ }if there is a nearring homomorphism $g:B\rightarrow N$
such that $f\circ g=1_{B}.$

\bigskip

\begin{proposition}
If $f:N\rightarrow B$ is a split epimorphism\textit{\ }with $\ker (f)=A,$
then $N$ is a semidirect sum of $A$ and a subnearring of $N$ which is an
isomorphic copy of $B$.
\end{proposition}

\begin{proof}
Let $g:B\rightarrow N$ be the right inverse of $f,$ i.e., $f\circ g=1_{B}.$
Then $g$ must necessarily be injective and hence $B$ is isomorphic to the
subnearring $g(B)$ of $N.$ Moreover, $A$ is an ideal of $N$ and since $f$ is
surjective, $N/A\cong B.$ We also have $A\cap g(B)=0$ by the definition of $%
A $ and $f\circ g=1_{B}.$ For any $n\in N,$ $g(f(n))\in g(B)$ and $%
n=(n-g(f(n)))+g(f(n))$ with $n-g(f(n))\in A.$ Thus $N=A+g(B).$
\end{proof}

\bigskip

\noindent In the category of nearrings, a split epimorphism is a projection
morphism of a semidirect sum. The next result gives the requirements to
construct a semidirect sum from two nearrings $A$ and $B.$ The proof follows
from Theorem 2.3.

\bigskip

\begin{proposition}
Let $A$ and $B$ be two nearrings with associated functions $E(\psi
,[-,-],\alpha ,\beta ,\left\langle -,-\right\rangle )$ fulfilling the
initial conditions specified in the previous section subject to the stronger
requirements $[-,-]=\left\langle -,-\right\rangle =0.$ The $E$-sum $%
A\boxtimes B$ is a semidirect sum of $A$ and $B$ if and only if the
conditions $(G_{2}),(A_{1}),(A_{2}),(D_{1})$ and $(D_{2})$ hold.
\end{proposition}

\bigskip

\noindent This semidirect sum of $A$ and $B$ in the proposition above is
actually called the \textit{outer semidirect sum} of $A$ and $B$ since $%
A\cong (A,0)$ which is an ideal of $A\boxtimes B.$ If necessary, this
semidirect sum will be denoted by $A\rtimes B(\psi ,\alpha ,\beta ).$ Note
that for $(a,b)\in A\rtimes B,$ $(a,b)=(a,0)+(0,b)=(0,b)+(a^{\prime },0)$
where $a^{\prime }=\psi _{-b}(a).$ The more restricted conditions required
in the previous result simplify the conditions $%
(G_{2}),(A_{1}),(A_{2}),(D_{1})$ and $(D_{2})$ to:

\noindent $(G_{2})$ $\psi _{b_{1}}\circ \psi _{b_{2}}=\psi _{b_{1}+b_{2}}$
for all $b_{1},b_{2}\in B,$ i.e., $\psi :(B,+)\rightarrow (Aut(A,+),\circ )$
is a group homomorphism;

\noindent $(A_{1})$ For all $a_{2},a_{3}\in A,b_{1},b_{2},b_{3}\in B,$

$\qquad \alpha _{b_{3}}(\beta _{b_{1},b_{2}}(a_{2}),a_{3})+\beta
_{b_{1}b_{2},b_{3}}(a_{3})=\beta _{b_{1},b_{2}b_{3}}(\alpha
_{b_{3}}(a_{2},a_{3})+\beta _{b_{2},b_{3}}(a_{3}));$

\noindent $(A_{2})$ For all $a_{1},a_{2},a_{3}\in A,b_{2},b_{3}\in B,$

\qquad $\alpha _{b_{3}}(\alpha _{b_{2}}(a_{1},a_{2}),a_{3})=\alpha
_{b_{2}b_{3}}(a_{1},\alpha _{b_{3}}(a_{2},a_{3})+\beta
_{b_{2},b_{3}}(a_{3}));$

\noindent $(D_{1})$ For all $b_{1},b_{2},b_{3}\in B,$

\qquad $\beta _{b_{1}+b_{2},b_{3}}=\beta _{b_{1},b_{3}}+\psi
_{b_{1}b_{3}}\circ \beta _{b_{2},b_{3}};$ and

\noindent $(D_{2})$ For all $a_{2},a_{3}\in A,b_{1},b_{3}\in B,$

\qquad $\alpha _{b_{3}}(\psi _{b_{1}}(a_{2}),a_{3})+\beta
_{b_{1},b_{3}}(a_{3})=\beta _{b_{1},b_{3}}(a_{3})+\psi _{b_{1}b_{3}}(\alpha
_{b_{3}}(a_{2},a_{3})).$

\bigskip

\noindent By Theorem 2.4 we know that any semidirect sum of two nearrings $A$
and $B$ is an $E$-sum $A\boxtimes B$ with associated functions $E(\psi
,[-,-],\alpha ,\beta ,\left\langle -,-\right\rangle ).$ This section is
concluded with the canonical way of doing so.

\bigskip

\begin{proposition}
Suppose the nearring $N$ is a \textit{semidirect sum} of the two nearrings $%
A $ and $B$. Then $N$ is equivalent to the $E$-sum $A\boxtimes B$ with
associated functions $E(\psi ,[-,-],\alpha ,\beta ,\left\langle
-,-\right\rangle )$ given by

$[b_{1},b_{2}]=\left\langle b_{1},b_{2}\right\rangle =0$ for all $%
b_{1},b_{2}\in B;$

$\psi _{b}(a)=b+a-b$ for all $b\in B,a\in A;$

$\alpha _{b}(a_{1},a_{2})=a_{1}(a_{2}+b)$ for all $b\in B,a_{1},a_{2}\in A;$
and

$\beta _{b_{1},b_{2}}(a)=b_{1}(a+b_{2})-b_{1}b_{2}$ for all $b_{1},b_{2}\in
B,a\in A.$
\end{proposition}

\begin{proof}
We know $A$ is an ideal of $N,$ $B$ is a subnearring of $N,$ $N=A+B$ and $%
A\cap B=0.$ Let $p:N\rightarrow B$ be the mapping $p(a+b)=b.$ This is a
surjective nearring homomorphism with $\ker (p)=A$ and its restriction to $B$
is the identity map. Let $\sigma :B\rightarrow N$ be the choice function $%
\sigma (b)=b$ for all $b\in B.$ This is well-defined since $p\circ \sigma
=1_{B}$, $p(b)=b$ and $\sigma (0)=0.$ For this choice of $\sigma ,$ the
functions above are exactly the functions defined in Step 2 of the previous
section and we can rest our case.
\end{proof}

\bigskip

\subsection{Examples}

\noindent We will often use the following. Let $G$ be a group written
additively and let $A$ and $B$ be nearrings. The nearring $M(G)$ is the
nearring of all self maps on $G$ with respect to pointwise addition and
composition of functions. Then $G$ is a left $M(G)$-group with respect to
the canonical action $(f,t)\mapsto f(t)$ for $f\in M(G),t\in G.$ Recall that
for a nearring $N,$ the notion $N$-group is the nearring terminology for a
left $N$-module. As usual $A^{G}:=\{f\mid f:G\rightarrow A\}$ denotes the
product of $\left\vert G\right\vert $ copies of the nearring $A$ and it is a
nearring with respect to componentwise addition and multiplication, i.e., $%
(f+g)(t)=f(t)+g(t)$ and $(fg)(t)=f(t)g(t)$ for $f,g\in A^{G},t\in G.$ The
set $End(A):=End((A,+,\cdot ))$ denotes all the endomorphisms of the
nearring $A$ and it is a semigroup with respect to operation $f\ast
g:=g\circ f$.

\bigskip

\textbf{3.1.1} Bellas [3] considers an action of the multiplicative
semigroup $(B,\cdot )$ on $A$ given by the semigroup homomorphism $\gamma
:(B,\cdot )\rightarrow (End(A),\ast ),$ $\gamma (b)=\gamma _{b}:A\rightarrow
A.$ This means, for all $a,a_{1},a_{2},\in A$ and $b,b_{1},b_{2}\in B:$

\qquad $(B_{1})$ $\gamma _{b_{2}}(\gamma _{b_{1}}(a))=\gamma
_{b_{1}b_{2}}(a),$

\qquad $(B_{2})$ $\gamma _{b}(a_{1}+a_{2})=\gamma _{b}(a_{1})+\gamma
_{b}(a_{2})$ and

\qquad $(B_{3})$ $\gamma _{b}(a_{1}a_{2})=\gamma _{b}(a_{1})\gamma
_{b}(a_{2}).$

\noindent It follows that $\gamma _{b}(0)=0$ and $\gamma _{b}(-a)=-\gamma
_{b}(a).$

\noindent For such an action $\gamma ,$ the Cartesian product $A\times B$ is
a nearring with respect to the operations

\qquad $(a_{1},b_{1})+(a_{2},b_{2})=(a_{1}+a_{2},b_{1}+b_{2})$ and

\qquad $(a_{1},b_{1})(a_{2},b_{2})=(\gamma _{b_{2}}(a_{1})a_{2},b_{1}b_{2}).$

\noindent This nearring is denoted by $A\times _{\gamma }B$ and Bellas calls
it the semidirect product of $A$ and $B.$ Here $A^{\ast }:=\{(a,0)\mid a\in
A\}$ is an ideal and $B^{\ast }:=\{(0,b)\mid b\in B\}$ a right ideal of $%
A\times _{\gamma }B.$ Moreover, $\frac{A\times _{\gamma }B}{A^{\ast }}\cong
B\cong B^{\ast }$ (nearring isomorphisms)$.$ In general, this construction
will not be a semidirect sum of $A$ and $B$ as defined above since the
nearring $A$ need not be isomorphic to $A^{\ast }$ (cf. Veldsman [13],
Proposition 3). If, in addition to $(B_{1}),(B_{2})$ and $(B_{3}),$ a fourth
condition

\qquad $(B_{4})$ $\gamma _{0}(a_{1})a_{2}=a_{1}a_{2}$ for all $%
a_{1},a_{2},\in A;$

\noindent is required, then $A\times _{\gamma }B=A\rtimes B(\psi ,\alpha
,\beta )$ where the parameters $\psi ,\alpha $ and $\beta $ are given by $%
\psi _{b}:=1_{A}$, $\alpha _{b}(a_{1},a_{2}):=\gamma _{b}(a_{1})a_{2}$ and $%
\beta (b_{1},b_{2})=0$ for all $b,b_{1},b_{2}\in B$ and $a_{1},a_{2}\in A.$

\noindent A particular case of the Bellas construction used for defining the
wreath product of two nearrings is: Let $(G,+)$ be a left $B$-group, let $N$
be a nearring with $A:=N^{G}.$ Define $\gamma :(B,\cdot )\rightarrow
(End(A),\ast )$ by $\gamma (b)=\gamma _{b}:A\rightarrow A$ with $\gamma
_{b}(f)=f_{b}:G\rightarrow N,f_{b}(t)=f(bt)$ for all $b\in B,f\in A$ and $%
t\in G.$ It can be verified that $\gamma $ is a well-defined semigroup
homomorphism. This semidirect product $A\times _{\gamma }B$ is called the
wreath product of the nearrings $N$ and $B$ with respect to $\gamma $ and is
denoted by $N$ $\underset{\gamma }{Wr}$ $B.$ In general it is not a
semidirect sum. For this, $\gamma _{0}(f_{1})f_{2}=f_{1}f_{2}$ for all $%
f_{1},f_{2}\in A$ must be required. This can be achieved by imposing certain
conditions on the nearring $N$ and/or by choosing $A$ to be a suitable
subnearring of $N^{G}.$

\noindent In conclusion, we should mention that this notion of the wreath
product of nearrings has been investigated by several authors under the name
wreath sum of nearrings; see, for example, Chowdhury and Nath [5] and
Chowdhury and Das [4].

\bigskip

\textbf{3.1.2} Let $G$ be a faithful $A$-group. Then $A$ can be embedded as
a subnearring in $M(G)$ via $a\mapsto \overline{a}:G\rightarrow G,\overline{a%
}(t)=at$ for all $t\in G.$ We identify $a$ and $\overline{a},$ hence every $%
a\in A$ is regarded as a function from $G$ to $G.$ Suppose $G$ is a $B$%
-group. Then $M(G)$ is also a $B$-group in a canonical way $(b,f)\mapsto bf,$
$bf:G\rightarrow G$ is defined by $(bf)(t):=bf(t)$ for all $b\in B,f\in M(G)$
and $t\in G.$ Choose a function $f$ from $M(G)$ and keep it fixed. Now $Bf$
is a $B$-subgroup of $M(G).$ We suppose that the function $f$ is such that
the following four conditions are satisfied for all $a,a_{1},a_{2},\in A,$ $%
b,b_{1},b_{2}\in B$ and $t\in G:$

\qquad $(i)$ $bf+a-bf\in A;$

\qquad $(ii)$ $a_{1}(a_{2}+bf)\in A;$

\qquad $(iii)$ $(b_{1}f)(a+b_{2}f)-b_{1}b_{2}f\in A;$ and

\qquad $(iv)$ $f(bt)=bf(t)$ and $f(f(t))=f(t).$

\noindent This determines a semidirect sum $A\rtimes B(\theta ,\alpha ,\beta
)$ where $\psi _{b}(a):=bf+a-bf,$ $\alpha _{b}(a_{1},a_{2})=a_{1}(a_{2}+bf)$
and $\beta _{b_{1},b_{2}}(a)=(b_{1}f)(a+b_{2}f)-b_{1}b_{2}f.$ The operations
are given by

\qquad $(a_{1},b_{1})+(a_{2},b_{2})=(a_{1}+b_{1}f+a_{2}-b_{1}f,b_{1}+b_{2})$
and

$\qquad
(a_{1},b_{1})(a_{2},b_{2})=((a_{1}+b_{1}f)(a_{2}+b_{2}f)-b_{1}b_{2}f,b_{1}b_{2}). 
$

\noindent We look at two special cases of this construction.

\noindent \textbf{(a)} The first is due to Betsch [1], but here we look at a
slightly altered version given in Veldsman [14]. Let $B=(%
\mathbb{Z}
,+,\cdot )$ be the ring of integers. Let $G$ be a group and let $f=1_{G}$,
the identity function on $G.$ Suppose that $A$ is a nearring such that $G$
is a faithful $A$-group and conditions $(i)$ and $(ii)$ are fulfilled.
Clearly with this choice of $f,$ $(iv)$ is satisfied. Next we show that $(i)$
is sufficient for $(iii):$ For $a\in A$ and $n,m\in 
\mathbb{Z}
,$ say $n>0,$

\qquad $(n1_{G})(a+m1_{G})-nm1_{G}$

\qquad $=a+\tsum\limits_{i=1}^{n-1}(m1_{G}+a)+[m1_{G}-nm1_{G}]$

\qquad $=a+\tsum%
\limits_{i=1}^{n-1}(im1_{G}+a-im1_{G})+[(n-1)m1_{G}+m1_{G}-nm1_{G}]$

\qquad $=a+\tsum\limits_{i=1}^{n-1}(im1_{G}+a-im1_{G})\in A.$

\noindent A similar argument takes care of the case when $n<0.$ In
conclusion, here we have a nearring $A$ with sufficient (and actually also
necessary) conditions to ensure that it can be embedded as an ideal in a
nearring $A\rtimes 
\mathbb{Z}
$ with identity. The operations are given by:

\qquad $(a_{1},n_{1})+(a_{2},n_{2})=(a_{1}+n_{1}f+a_{2}-n_{1}f,n_{1}+n_{2}),$

$\qquad
(a_{1},n_{1})(a_{2},n_{2})=((a_{1}+n_{1}f)(a_{2}+n_{2}f)-n_{1}n_{2}f,n_{1}n_{2}) 
$

\noindent and the identity element is $(0,1).$ When $A$ is a ring, then $%
A\rtimes 
\mathbb{Z}
$ is just the Dorroh extension of $A$ (i.e., the canonical unital extension
of the ring $A).$

\bigskip

\noindent \textbf{(b)} Suppose $G$ is a $0$-symmetric nearring which
contains a normal subgroup $I$ and an idempotent element $e\in G$ such that $%
Ie\subseteq I.$ Define a function $f:G\rightarrow G$ by $f(t):=te$ for all $%
t\in G.$ Now $A:=\{f\in M(G)\mid f(G)\subseteq I\}$ is a subnearring of $%
M(G) $ and $G$ is a faithful $A$-group. With $B=(%
\mathbb{Z}
,+,\cdot )$, the four conditions above are fulfilled. Indeed, $(i)$ and $%
(ii) $ are valid by the definition of $A$ and the normality of $I,$ $(iii)$
can be done as in $(a)$ above using the assumption $Ie\subseteq I$ and $(iv)$
follows directly from the definition of $f.$

\bigskip

\noindent The general radical theory of nearrings contains many
constructions of purpose build nearrings. Some of these are semidirect sums
and a select few are given as the next examples.

\bigskip

\textbf{3.1.3 }(cf. Veldsman [13], Theorem 2.5) Let $N$ be a nearring and
let $I:=(0:N)_{N}.$ Let $A$ be the subnearring of $N^{N}$ defined by $%
A:=\{f\in N^{N}\mid $ for all $x,y\in N,f(x)-f(y)\in I\}$ and let $B$ be the
nearring with zero multiplication defined on the group $(N,+).$ Define $\psi
:(B,+)\rightarrow (Aut(A,+),\ast )$ by $\psi (b)=\psi _{b}:A\rightarrow A$
and $\psi _{b}(f):=f_{b}:N\rightarrow N$ is given by $f_{b}(x):=f(x-b)$ for
all $b\in B,f\in A$ and $x\in N.$ For $\alpha :B\rightarrow Map(A^{2},A)$ we
let $\alpha (b)=\alpha _{b}:A^{2}\rightarrow A$ be given by $\alpha
_{b}(f,g):=fg$ for all $b\in B$ and $f,g\in A.$ Lastly, $\beta
:B^{2}\rightarrow Map(A,A)$ is defined by $\beta (b_{1},b_{2})=0$ for all $%
b_{1},b_{2}\in B.$ Here $(f,b_{1})+(g,b_{2})=(f+\psi
_{b_{1}}(g),b_{1}+b_{2}) $ and $(f,b_{1})(g,b_{2})=(fg,0).$

\bigskip

\textbf{3.1.4} (cf. Betsch and Kaarli [2]) Let $N$ and $M$ be nearrings,
both with multiplication zero (i.e., for all $a,b\in N$ (or from $M),$ $%
ab=0).$ Let $A=N\oplus M$ be the direct sum of $N$ and $M$ and let $B=N.$ We
write $A\times B$ as $A\times B=(N,M,N).$ Let $\psi _{b}=1_{A}$ and $\alpha
_{b}=0$ for all $b\in B.$ Define $\beta :B^{2}\rightarrow Map(A,A)$ by $%
\beta _{b_{1},b_{2}}(a,x)=\left\{ 
\begin{array}{l}
(b_{1},0)\text{ if }x\neq 0 \\ 
(0,0)\text{ if }x=0%
\end{array}%
\right. .$ This gives a semidirect sum on $A$ and $B$ with componentwise
addition and the multiplication is given by:

\qquad $(a,x,b)(c,y,d)=\left\{ 
\begin{array}{l}
(b,0,0)\text{ if }y\neq 0 \\ 
(0,0,0)\text{ if }y=0%
\end{array}%
\right. .$ Here $(N,M,0)\lhd (N,M,N)$ with $\frac{(N,M,N)}{(N,M,0)}\cong
N\cong (0,0,N).$

\noindent There are many more constructions on this theme that give
semidirect sums, but they will not be mentioned here.

\bigskip

\textbf{3.1.5} (cf. Veldsman [15]) Let $(N,+)$ be a group and let $N^{0}$ be
the nearring on $(N,+)$ with zero multiplication $(ab=0$ for all $a,b)$ and $%
N^{c}$ the nearring with constant multiplication $(ab=a$ for all $a,b).$ Let 
$A$ be the nearring defined by $A=\{(a,b)\mid a,b\in N\}$ with

\qquad $(a,b)+(c,d)=(a+c,-c+b+c+d)$ and

$\qquad (a,b)(c,d)=(a,0).$

\noindent Let $B=N^{c}.$ For $x,y\in B,$ let $\psi _{x}:A\rightarrow A$ be $%
\psi _{x}=1_{A}$ and let $\beta (x,y)=0.$ For $x\in B$ and $(a,b),(c,d)\in
A, $ let

\qquad \qquad \qquad $\alpha _{x}((a,b),(c,d))=\left\{ 
\begin{array}{l}
(a+b,0)\text{ if }x\neq 0 \\ 
(a,0)\text{ if }x=0%
\end{array}%
\right. .$

\noindent All this defines a semidirect sum $A\rtimes B(\psi ,\alpha ,\beta
) $ and, if we write $A\rtimes B$ as $(N,N,N),$ the operations are given by:

\qquad \qquad $(a,b,x)+(c,d,y)=(a+c,-c+b+c+d,x+y)$ and

\qquad \qquad $(a,b,x)(c,d,y)=\left\{ 
\begin{array}{l}
(a+b,0,x)\text{ if }y\neq 0 \\ 
(a,0,x)\text{ if }y=0%
\end{array}%
.\right. $

\bigskip

\noindent \textbf{References}

\noindent \lbrack 1] G. Betsch (1987). Embedding of a Near-Ring into a
Near-Ring with Identity, \textit{North-Holland Mathematics Studies} Volume 
\textbf{137}, 37-40.

\noindent \lbrack 2] G. Betsch and K. Kaarli (1985). Supernilpotent radicals
and hereditariness of semisimple classes, \textit{Coll. Soc. J. Bolyai, 38,
Theory of radicls, Eger, 1982}, North Holland, 47-58.

\noindent \lbrack 3] Carlos Ruiz de Velasco y Bellas (1983). Wreath products
of near-rings, \textit{Houston Journal of Mathematics}, \textbf{9} no 3, 357
- 362.

\noindent \lbrack 4] Khanindra Chandra Chowdhury and Prohelika Das (2012).
Wreath sum of near-rings and near-ring groups, S\textit{outheast Asian
Bulletin of Mathematics} \textbf{36}, 169-185.

\noindent \lbrack 5] K.C. Chowdhury and Dipti Nath (2013). Some aspects of $%
\theta $-wreath sum of near-rings, \textit{Far East Journal of Mathematical
Sciences} \textbf{75} (1), 27-46.

\noindent \lbrack 6] J. R. Clay (1992). \textit{Nearrings: Geneses and
Applications}, Oxford University Press.

\noindent \lbrack 7] Alberto Facchini and David Stanovsk\'{y}. \textit{%
Semidirect products in Universal Algebra}, arXiv:2311.04321v1,
https://doi.org/10.48550/arXiv.2311.04321.

\noindent \lbrack 8] C.C. Ferrero and G. Ferrero (2002). \textit{Nearrings:
some developments linked to semi-groups and groups, }Kluwer Academic
Publishers.

\noindent \lbrack 9] J.D.P. Meldrun (1985). \textit{Near-Rings and Their
Links with Groups}, Pitman Advanced Publishing Program, Boston.

\noindent \lbrack 10] M. Petrich (1985). Ideal extensions of rings, \textit{%
Acta Mathematica Hungarica} \textbf{45}, 263-283.

\noindent \lbrack 11] G. Pilz (1983). \textit{Near-Rings: The Theory and Its
Applications}, North Holland Mathematics Studies (revised edition).
Amsterdam.

\noindent \lbrack 12] L. R\'{e}dei (1967). \textit{Algebra} Vol I, Pergamon
Press, Oxford-New York.

\noindent \lbrack 13] S. Veldsman (1991). An overnilpotent radical theory
for near-rings, \textit{Journal of Algebra} \textbf{144}, 248-265.

\noindent \lbrack 14] S. Veldsman (1992). On unital extensions of near-rings
and their radicals, \textit{Math. Pannonica} \textbf{3}, 77-81.

\noindent \lbrack 15] S. Veldsman (1996). The general theory of near-rings -
answers to some open problems, \textit{Algebra Universalis} \textbf{36}, 185
- 189.

\end{document}